\documentclass[a4,11pt]{article}

\usepackage{amsmath,amsbsy,amsfonts,amssymb}
\usepackage{color} 

\newtheorem{theorem}{Theorem}    
\newtheorem{remark}{Remark}   
\newtheorem{lemma}[theorem]{Lemma}
\newtheorem{corollary}{Corollary}

\newcommand{\pp}{\mathfrak{p}}
\newcommand{\qq}{\mathfrak{q}}
\newcommand{\Va}{\mathfrak{a}}
\newcommand{\ff}{\mathfrak{f}}
\newcommand{\ee}{\mathfrak{e}}
\newcommand{\hh}{\mathfrak{h}}

\renewcommand{\Re}{\mathop{\rm Re}}
\renewcommand{\Im}{\mathop{\rm Im}}

\newcommand{\boxend}{{\hfill $\rule{2.0mm}{2.0mm}$}}

\begin{document}

\title{Extended relativistic Toda lattice, L-orthogonal polynomials and associated Lax pair\thanks{The first and third authors are supported by funds from FAPESP (2014/22571-2, 2016/09906-0, 2017/12324-6) and CNPq (305073/2014-1, 305208/2015-2, 402939/2016-6) of Brazil. The second author was supported by grant from CAPES of Brazil.}}

\author
{
 {Cleonice F. Bracciali$^{1,}$, \ Jairo S. Silva$^2$, \ A. Sri Ranga$^1$\thanks{cleonice@ibilce.unesp.br; jairo.santos@ufma.br; ranga@ibilce.unesp.br}}
  \\[0.5ex]
 {\small $^1$Depto de Matem\'{a}tica Aplicada, IBILCE, } \\
 {\small UNESP - Univ Estadual Paulista,} 
 {\small 15054-000, S\~{a}o Jos\'{e} do Rio Preto, SP, Brazil. } \\[0.5ex]
{\small $^2$Depto de Matem\'{a}tica, Universidade Federal do Maranh\~{a}o, 65080-805, S\~{a}o Lu\'{\i}s, MA, Brazil.}
  }

\maketitle

\begin{abstract}

When a measure $\Psi(x)$ on the real line is  subjected to the modification $d\Psi^{(t)}(x) = e^{-tx} d \Psi(x)$, then  the coefficients of the recurrence relation of the orthogonal polynomials in $x$ with respect to the measure $\Psi^{(t)}(x)$ are known to satisfy the so-called Toda lattice formulas as  functions of $t$.
In this paper we consider a modification of the form $e^{-t(\pp x+\qq/x)}$ of measures or, more generally, of moment functionals, associated with orthogonal L-polynomials and  show that the coefficients of the recurrence relation of these  L-orthogonal polynomials satisfy what we call an extended relativistic Toda lattice. Most importantly, we also establish the so called Lax pair representation associated with this extended relativistic Toda lattice. These results also cover the (ordinary) relativistic Toda lattice formulations considered in the literature by assuming either $\pp =0$ or $\qq=0$. However, as far as Lax pair representation is concern, no complete Lax pair representations were established for the respective relativistic Toda lattice formulations.  
Some explicit examples of extended relativistic Toda lattice and Langmuir lattice are also presented. As further results, the lattice formulas that follow from the three term recurrence relations associated with  kernel polynomials on the unit circle are also established. \\

{\noindent}Keywords:  Relativistic Toda lattice; Lax pairs; L-orthogonal polynomials; Kernel polynomials on the unit circle   

{\noindent}MSC:  34A33, 42C05, 33C47, 47E05  

\end{abstract}

\section{Introduction}  
\label{Sec-Intro}

Let $\Psi$ be a positive measure defined on the real line, and let $\{P_n\}_{n=0}^{\infty}$ be the sequence of monic orthogonal polynomials with respect to $\Psi$, in the sense that
\[ 
\int_{\mathbb{R}} x^sP_n(x)d\Psi(x)=0, \quad 0\leq s\leq n-1, 	\quad n\geq1,
\]
where $P_n$ is a polynomial of exact degree $n$.
It is known (see, for example, \cite{Chihara-book,Ism2005,Szego-book-1939}) that these polynomials satisfy the following three term recurrence relation
\begin{equation}\label{relrec}
P_{n+1}(x)=(x-b_{n+1})P_n(x)-a_{n+1}P_{n-1}(x) ,\quad n\geq1,
\end{equation}
with $P_0(x)=1$ and $P_1(x)=x-b_1$. Moreover, for any $n\geq1$,  the coefficients $b_{n}$ are real and  $a_{n+1}$ are positive ($a_1$ is arbitrary).

Consider the modified measure ${\Psi}^{(t)}$,  in one time variable $t$,  given by $ d {\Psi}^{(t)}(x)$ $ = e ^{-t x}  d \Psi(x)$, and the associated monic orthogonal polynomials $P_n(x; t)$.  As given in \cite{Ism2005} (see also \cite{Peh2001}) the recursion coefficients $a_{n}(t)$ and $b_{n}(t)$ which appear in the three term recurrence relation (\ref{relrec}) for $\{P_n(x; t)\}_{n \geq 0}$  satisfy the semi-infinite Toda lattice equations of motion 
\begin{equation} \label{eqtod}
\begin{array}{l}
 \dot{a}_n(t) =  a_n(t)[b_{n-1}(t)-b_{n}(t)], \\[1.5ex]
 \dot{b}_n(t) =  a_n(t)-a_{n+1}(t),
\end{array}
\quad n \geq 1,
\end{equation}
with the initial conditions $b_{0}(t)=1$, $a_{1}(t)=0$, $a_n(0)=a_n$
and $b_n(0)=b_n$. Here,  we use the usual notation $\dot{f} = \frac{d}{d t} f$. 

Toda lattice is a system of particles on the line with exponential interaction of nearest neighbours \cite{Suris2003}. Toda was the first to consider such a system for infinitely many particles on the line \cite{to1}. The Toda lattice equations (\ref{eqtod}) are obtained from the Newtonian equations of motion (see, for example, \cite{Suris2003})
\begin{equation}\label{todaj}
\ddot{x}_n=e^{x_{n-1} - x_{n}}-e^{x_{n} - x_{n+1}},\quad n\geq1,
\end{equation}
when one takes $b_n=\dot{x}_n$  and $a_n=e^{x_{n-1} - x_{n}}$ for $n\geq1$. \\[-1ex]

The present work  considers a similar study in the case of  the so-called L-orthogonal polynomials. In this respect,  let   $\mathcal{L}$  be a moment functional defined on the linear space $Span\{1, x^{-1}, x, $ $x^{-2}, \ldots \}$ of Laurent polynomials. 

Given $\pp, \qq \in \mathbb{C}$, we also assume that the moment functional $\mathcal{L}$ is  such that $\mathcal{L}[e^{-t\left(\pp x + \qq/x\right)}x^{k}]$ exists for all $k \in \mathbb{Z}$  and for all $t \geq 0$, and that 
\[
     \frac{d}{dt}\mathcal{L}\left[e^{-t\left(\pp x + \frac{\qq}{x}\right)}f(x,t)\right] = \mathcal{L}\left[\frac{\partial}{\partial t} \left( e^{-t\left(\pp x + \frac{\qq}{x} \right)}f(x,t)\right)\right],
\]
for any $f(x,t)$ which is a Laurent polynomial in $x$ and an absolutely continuous function in $t$ for $t \geq 0$. 

Starting with the above $\mathcal{L}$,  we consider  the  parametric family of moment functionals $\mathcal{L}^{(t)}$, $t \geq 0$, such that   
\begin{equation} \label{Eq-MF-TodaModification}
\mathcal{L}^{(t)}[x^{k}] \ = \ \mathcal{L}\left[e^{-t \left(\pp x + \frac{\qq}{x} \right)}x^{k}\right] \ = \ \nu_{k}^{(t)}, \quad k=0, \pm1, \pm2, \ldots, \quad t \geq 0,
\end{equation}
and assume also that 
\begin{equation}  \label{Eq-newHankelCond}
(a) \ \ H_{n}^{(-n)}(t) \neq 0 \quad \mbox{and} \quad  (b)\ \ H_{n+1}^{(-n)}(t) \neq 0,
\end{equation}
for $n \geq 0$ and $t \geq 0$, where the associated Hankel determinants $H_{n}^{(m)}(t)$ are given by $H_{0}^{(m)}(t)= 1$ and
\[
H_{n}^{(m)}(t) = \left|
\begin{array}{cccc}
\nu_{m}^{(t)}     & \nu_{m+1}^{(t)} & \cdots & \nu_{m+n-1}^{(t)}   \\
\nu_{m+1}^{(t)}   & \nu_{m+2}^{(t)} & \cdots & \nu_{m+n}^{(t)}     \\
\vdots            & \vdots          &        & \vdots              \\
\nu_{m+n-1}^{(t)} & \nu_{m+n}^{(t)} & \cdots & \nu_{m+2n-2}^{(t)}  \\
\end{array} \right|, \quad n,m \in \mathbb{Z}, \quad n \geq 1. 
\]

We now define, for any fixed $t \geq 0$, the sequence $\{Q_{n}(x;  t)\}_{n \geq 0}$ of polynomials in $x$   by 
\begin{equation} \label{Eq-Def-Lorthogonality} 
\begin{array}{l}
     Q_{n}(x; t) \mbox{ is a monic polynomial of degree } n \mbox{ in } x  \\[2ex]
     \mathcal{L}^{(t)} \left[ x^{-n+s}Q_{n}(x; t) \right] = 0, \quad s =0, 1, \ldots, n-1,
\end{array}
\end{equation}
for $n \geq 1$. It is known (see, for example,  \cite{JNT83,JTW1980,SRM96})  that the existence of these polynomials is assured by the conditions in (\ref{Eq-newHankelCond}). Also when  the conditions in (\ref{Eq-newHankelCond}) hold these polynomials are known to satisfy the following three term recurrence relation 
\begin{equation} \label{eqrrttj}
Q_{n+1}(x; t) = [x-\beta_{n+1}(t)] Q_{n}(x; t) - \alpha_{n+1}(t)\, x Q_{n-1}(x; t), \quad n \geq 1,
\end{equation}
with $Q_{0}(x; t)=1, \ Q_{1}(x; t)=x - \beta_{1}(t)$. The known expressions for the  coefficients $\beta_n(t)$ and $\alpha_n(t)$ in terms of the moment functional $\mathcal{L}^{(t)}$ are presented in  Section \ref{Sec-Proof-TheoToda} of this manuscript. 

The polynomials  $Q_{n}(x; 0)$ (to be more general, the polynomials $Q_{n}(x; t)$ for any fixed $t$) have been referred  to as L-orthogonal polynomials in some previous papers in the subject  (see \cite{FelixRangaVeron-ANM2011} and references therein). Hence, throughout in this manuscript we will refer to the polynomials $Q_{n}(x; t)$ as L-orthogonal polynomials with respect to the moment functional $\mathcal{L}^{(t)}$ or simply L-orthogonal polynomials.  However, it is important to mention that Zhedanov in his many contributions (see, for example, \cite{Zhed98})  refers to  such polynomials as Laurent biorthogonal polynomials. 

The  polynomials $Q_{n}(x; 0)$, when the determinants in (\ref{Eq-newHankelCond}) (for $t=0$) are positive, have played an important role in the study of the so-called strong Stieltjes moment problem which was introduced in \cite{JTW1980}. 

With respect to the present objective, it is known that the recursion coefficients $\beta_{n}(t)$ and $\alpha_{n}(t)$ (see, for example, \cite{CKVA2002,KhaMirZhed-IJMP1997}) satisfy the equations of motion
\begin{equation}  \label{RTL1}
\begin{array}{l}
\displaystyle \dot{\beta}_{n} = \beta_{n} \left( \frac{\alpha_{n+1}}{\beta_{n+1}\beta_{n}} - \frac{\alpha_{n}}{\beta_{n}\beta_{n-1}} \right),  \quad
\displaystyle  \dot{\alpha}_{n} = \alpha_{n} \left( \frac{1}{\beta_{n-1}} - \frac{1}{\beta_{n}} \right),  
\end{array}  
\end{equation}
for the case where $\pp=0$ and $\qq=1$, and
\begin{equation}  \label{RTL2}
\begin{array}{l}
\displaystyle   \dot{\beta}_{n}  = \beta_{n} \left( \alpha_{n} - \alpha_{n+1} \right),  \quad
\displaystyle  \dot{\alpha}_{n}  =  \alpha_{n} \left( \alpha_{n-1} + \beta_{n-1} - \alpha_{n+1}  - \beta_{n}\right),
\end{array} 
\end{equation}
for the case where $\pp=1$ and $\qq=0$, both with $\beta_n(t) \neq 0$, $\alpha_{n+1}(t) \neq 0$ for $n \geq 1$ and $\alpha_1(t)= 0$  (further, $\alpha_{N+1}(t)=0$ in the finite case, i.e., $n=1, 2, \ldots, N$). The system of equations (\ref{RTL2}) has also been considered in \cite{Com92,ComHaf92}, where  the authors showed its connection to  T--fractions (or the equivalent M--fractions). Observe that the recurrence formula (\ref{eqrrttj}) is the three term recurrence relation satisfied by the denominator polynomials of M--fractions.

The above system of equations known nowadays as the {\em relativistic Toda lattice} was introduced by Ruijsenaars \cite{Rui1990} (and studied, for example, in 
\cite{BR88,BR89,Suris1996,Suris1997,Suris2003}) in the form of the following Newtonian equations of motion
\begin{equation} \label{eq_mov_suris11}
\ddot{x}_n =  (1 + g\dot{x}_{n+1}) (1 + g \dot{x}_{n}) \frac{e^{x_{n+1} - x_{n}}}{1+ g^2 e^{x_{n+1} - x_{n}}}  - (1 + g \dot{x}_{n}) (1 + g \dot{x}_{n-1}) \frac{e^{x_{n} - x_{n-1}}}{1+ g^2 e^{x_{n} - x_{n-1}}},
\end{equation}
where $g$ is a (small) parameter of the model having physical meaning as the inverse speed of light. Notice that the Newtonian equations of motion (\ref{eq_mov_suris11}) for the relativistic Toda lattice are, actually, a one-parameter perturbation of the Newtonian equations of motion (\ref{todaj}) for the usual Toda lattice.

In the present manuscript, different from what is considered in \cite{CKVA2002,KhaMirZhed-IJMP1997},  we consider the Toda lattice equations that follow from the two directional modification  $\pp x + \qq/x$ in (\ref{Eq-MF-TodaModification}).

The main purpose of the present manuscript is to show that the recursion  coefficients $\beta_{n}(t)$ and $\alpha_{n}(t)$  satisfy the lattice equations that we call the {\it extended relativistic Toda lattice} equations. We also establish the Lax pair formulations for this extended relativistic Toda lattice.
The main results of this paper are stated in the following  two theorems.

\begin{theorem}\label{Thm-toda-general} Consider a  moment functional $\mathcal{L}^{(t)}$, defined as in $(\ref{Eq-MF-TodaModification})$, and the  L-orthogonal polynomials $Q_n(x; t)$, $n\geq1$, defined in ($\ref{Eq-Def-Lorthogonality}) $. 
The recursion coefficients $\beta_{n}(t)$ and $\alpha_{n}(t)$ in the recurrence relation $(\ref{eqrrttj})$ satisfy the extended relativistic Toda lattice equations
\begin{equation} \label{Eq-Toda-02}
\dot{\beta}_{n} =  \pp\, \beta_{n} \left( \alpha_{n} - \alpha_{n+1} \right) +
\qq\, \beta_{n} \left(\frac{\alpha_{n+1}}{\beta_{n+1}\beta_{n}} - \frac{\alpha_{n}}{\beta_{n}\beta_{n-1}} \right)
\end{equation}
and
\begin{equation} \label{Eq-Toda-01}
\dot{\alpha}_{n} =  \pp\,  \alpha_{n} \left( \alpha_{n-1} + \beta_{n-1}-\alpha_{n+1} - \beta_{n} \right)
+ \qq\,  \alpha_{n} \left( \frac{1}{\beta_{n-1}} - \frac{1}{\beta_{n}} \right),
\end{equation}
for $n \geq 1$, with the initial conditions $\beta_{0}(t)= 1$, $\alpha_{0}(t)= -1$  and $\alpha_{1}(t)= 0$. Moreover,
\begin{equation} \label{Eq-Toda-03}
\dot{\gamma}_{n} = \pp\,\left(\alpha_{n}\gamma_{n} - \alpha_{n+1} \gamma_{n+1}\right) 
+ \qq\, \left( \frac{\alpha_{n+1}}{\beta_{n}} - \frac{\alpha_{n}}{\beta_{n-1}} \right), \quad n\geq1,
\end{equation}
with $\gamma_{n} = \alpha_{n+1} + \beta_{n}$, \ $n\geq1$. Here,  we have omitted the time variable $t$ for simplification.
\end{theorem}
The proof of this theorem is given in Section \ref{Sec-Proof-TheoToda} of the manuscript. 

The following interesting result holds with respect to the extended relativistic Toda lattice equations given by  Theorem \ref{Thm-toda-general}.

\begin{theorem} \label{Thm-Lpair-ETLE}
Let the infinite Hessenberg matrix $\mathcal{H}(t) = \mathcal{H}$ and the infinite tridiagonal matrix $\mathcal{F}(t) = \mathcal{F}$ be given by
\begin{equation} \label{Eq-LaxPairHF}
{\mathcal H}= \left(
\begin{array}{cccccc}
\gamma_1 & \gamma_2 & \gamma_3 & \cdots  & \gamma_{n} & \cdots \\[1ex]
\alpha_2   & \gamma_2 & \gamma_3 & \cdots  & \gamma_{n} & \cdots \\[1ex]
0        & \alpha_3   & \gamma_3 & \cdots  & \gamma_{n} & \cdots \\
\vdots   & \ddots   & \ddots   & \ddots  & \vdots     &        \\
0        & \cdots   & 0        & \alpha_{n}& \gamma_{n} & \cdots \\
\vdots   &          & \vdots   & \ddots  & \ddots     & \ddots
\end{array}\right),  \quad
{\mathcal{F}} = \left(
\begin{array}{cccccc}
\ff_1  & \hh_1  & 0      & \cdots & 0        & \cdots \\
\ee_2  & \ff_2  & \hh_2  & \ddots & \vdots   &        \\
 0     & \ee_3  & \ff_3  & \ddots & 0        & \cdots \\
\vdots & \ddots & \ddots & \ddots & \hh_{n-1}& \ddots \\
 0     & \cdots & 0      & \ee_n  & \ff_n    & \ddots \\
\vdots &        &\vdots  & \ddots & \ddots   & \ddots 
\end{array}\right), \qquad
\end{equation} 
where $\gamma_{k} = \alpha_{k+1} + \beta_{k}$,  
\[
  \ee_{k} = -\pp \alpha_{k}, \quad \ff_{k} = \pp \alpha_{k} + \frac{\qq}{\gamma_{k} - \alpha_{k+1}}, \quad \hh_{k} = -\frac{\qq}{\gamma_{k} - \alpha_{k+1}}, 
\]
for $k \geq 1$.  Then the pair $\{{\mathcal H}, {\mathcal F}\}$ is a Lax pair for the extended relativistic Toda lattice equations $(\ref{Eq-Toda-01})$ and $(\ref{Eq-Toda-03})$ given by  Theorem \ref{Thm-toda-general}.  Precisely,  
$ \dot{{\mathcal H}}=[{\mathcal H}, {\mathcal F}],$
where $[{\mathcal H}, {\mathcal F}]={\mathcal H}{\mathcal F}-{\mathcal F}{\mathcal H}.$
\end{theorem}

Further information about {\em Lax pairs} and the proof of Theorem \ref{Thm-Lpair-ETLE} are given  in Section \ref{Sec-Lax-Lorth} of the manuscript.
 
In Section \ref{Sec-LOrthogonal-RL} we consider the case in which $\mathcal{L}[f] = \int_{a}^{b} f(x) d \psi(x)$, where $0 \leq  a < b \leq \infty$ and $\psi$ is a strong positive measure defined on $[a, b]$.  The word strong, adopted from the work of Jones, Thron and Waadeland \cite{JTW1980} on the strong moment problem, is to indicate that the measure has moments of non-negative and negative order. Two examples with $\pp \qq \neq 0$ are given, where the solutions have very nice representations.   

The two parameter extension   $\pp x + \qq/x$ also permits us to extend the results to the unit circle. If $\mu$ is a positive measure on the unit circle  $\mathbb{T}=\{x=e^{i\theta} : 0 \leq \theta \leq 2\pi \}$, then with $\pp = \overline{\qq}$, the measures $\mu^{(t)}$ given by 
\[
    d \mu^{(t)}(x) = e^{-t(\pp x + \qq/x)} d \mu(x),
\]
is also positive in $\mathbb{T}$ for any $t$ real. Hence, respectively in Sections \ref{Sec-CDKernels-UC} and \ref{Sec-OPUC}, we have been able to proceed with the study of looking at the recurrence coefficients $\beta_{n}(t)$  and $\alpha_{n+1}(t)$ that follow from the moment functionals
\begin{equation} \label{Eq-Toda-Mod-KOPUC}
   \mathcal{L}^{(t)}[f] = \int_{\mathbb{T}} f(x)\, e^{-t(\overline{\qq} x + \qq/x)} (x-w)  d \mu(x),
\end{equation}
where $|w|=1$ fixed and from the moment functionals 
\begin{equation} \label{Eq-Toda-Mod-OPUC}
   \mathcal{L}^{(t)}[f] = \int_{\mathbb{T}} f(x)\,  e^{-t(\overline{\qq} x + \qq/x)}  x d \mu(x).
\end{equation}
In the case of (\ref{Eq-Toda-Mod-OPUC}) the polynomials $Q_{n}(x;t)$   are the monic orthogonal polynomials on the unit circle with respect to $\mu^{(t)}$ (also known as Szeg\H{o} polynomials), and in the case of (\ref{Eq-Toda-Mod-KOPUC})  the polynomials $Q_{n}(x;t)$ are the monic kernel polynomials, say $const\, K_{n}^{(t)}(x,w)$, on the unit circle with respect to $\mu^{(t)}$.  For a comprehensive collection of information concerning orthogonal polynomials on the unit circle and associated kernel polynomials we refer to \cite{Simon-book-p1}.


\setcounter{equation}{0}
\section{The proof of Theorem \ref{Thm-toda-general} }
\label{Sec-Proof-TheoToda}

Let $Q_{n}(x; t) = \sum_{j=0}^{n} a_{n,j}(t) x^j$, where $a_{n,n}(t) = 1$. From the  linear system (\ref{Eq-Def-Lorthogonality}) in the coefficients $a_{n,j}(t)$ of $Q_{n}(x; t)$, one easily finds 
\[
  Q_{n}(x; t) = \frac{1}{H_{n}^{(-n)}(t)}  \left|
\begin{array}{cccc}
\nu_{-n}^{(t)} & \nu_{-n+1}^{(t)} & \cdots & \nu_{0}^{(t)}   \\
\vdots         & \vdots           &        & \vdots          \\
\nu_{-1}^{(t)} & \nu_{0}^{(t)}    & \cdots & \nu_{n-1}^{(t)} \\
1              & x                & \cdots & x^n             \\
\end{array} \right|,  \quad n \geq 1,
\]
and if $\sigma_{n,-1}(t) = \mathcal{L}^{(t)}[x^{-n-1} Q_{n}(x; t)]$ and $\sigma_{n,n}(t) = \mathcal{L}^{(t)}[Q_{n}(x; t)]$ then 
\[
   \sigma_{n,-1}(t) = (-1)^{n}\frac{H_{n+1}^{(-n-1)}(t)}{H_{n}^{(-n)}(t)}, \quad \sigma_{n,n}(t) = \frac{H_{n+1}^{(-n)}(t)}{H_{n}^{(-n)}(t)}, \quad n \geq 1.
\]
Also observe that  $\sigma_{0,0}(t) = \nu_{0}^{(t)}$ and $\sigma_{0,-1}(t) = \nu_{-1}^{(t)}$.

Clearly, the sequence of polynomials $\{Q_{n}(x; t)\}_{n \geq 0}$ exists for all $t \geq 0$ because $H_{n}^{(-n)}(t) \neq 0$ for all $t \geq 0$, which follows from condition $(a)$ of (\ref{Eq-newHankelCond}). Observe that the condition $(b)$ in  (\ref{Eq-newHankelCond}) also means that $Q_{n}(0; t) \neq 0$ for all $n \geq 1$ and  for all $t \geq 0$. 

With both conditions $(a)$ and $(b)$ in  (\ref{Eq-newHankelCond}) the recurrence relation (\ref{eqrrttj}) holds for any $t \geq 0$, with  
$\beta_{1}(t) = \sigma_{0,0}(t)/\sigma_{0,-1}(t)$,
\begin{equation} \label{Eq-coefs-TTRR}
\beta_{n+1}(t) = - \alpha_{n+1}(t) \frac{\sigma_{n-1,-1}(t)}{\sigma_{n,-1}(t)} \ \ \mbox{and} \ \ \alpha_{n+1}(t) = \frac{\sigma_{n,n}(t)}{\sigma_{n-1,n-1}(t)}, \ \ n \geq 1.
\end{equation}
From the recurrence relation (\ref{eqrrttj}), we directly obtain
\begin{equation}\label{an0}
 a_{n,0}(t) = Q_{n}(0; t) = (-1)^n  \beta_{n} (t)\beta_{n-1} (t)\cdots \beta_{2}(t) \beta_{1}(t), \quad n \geq 1 
\end{equation}
and
\begin{equation} \label{abzeros}
\alpha_{n+1}(t) + \beta_{n+1}(t) = a_{n,n-1}(t) - a_{n+1,n}(t), \quad n\geq1.
\end{equation}

Notice that using (\ref{Eq-coefs-TTRR}), we also obtain for $n\geq0$,
\begin{equation} \label{sigxi}
\sigma_{n,n}(t) \  = \ \alpha_{n+1}(t)  \alpha_{n}(t) \cdots  \alpha_{2}(t) \sigma_{0,0}(t)
\end{equation}
and
\begin{equation} \label{sigman}
\sigma_{n,-1}(t) = (-1)^n \frac{\alpha_{n+1}(t) \alpha_{n} (t)\cdots \alpha_{2}(t) \sigma_{0,0}(t)}{\beta_{n+1}(t)\beta_{n}(t) \cdots  \beta_{2}(t)\beta_{1}(t)}
= \frac{\sigma_{n,n}(t)}{\beta_{n+1}(t)a_{n,0}(t)}.
\end{equation}

\begin{lemma}\label{Lemma-01} Let us denote $\tau_n(t) = \mathcal{L}^{(t)} \left[x Q_{n}(x; t) \right]$, $n\geq0$. Then,
\[
\tau_n(t) = \sigma_{n,n}(t) \sum_{k=1}^{n+1} [\alpha_{k+1}(t) + \beta_{k}(t)] =   \sigma_{n,n}(t) \sum_{k=1}^{n+1} \gamma_{k}(t),
\]
for $n\geq 0$, where $\gamma_{k}(t) = \alpha_{k+1}(t) + \beta_{k}(t)$, $k = 1, 2, \ldots, n+1.$
\end{lemma}
\begin{proof} 
The validity of the Lemma for $n=0$ is easily verified. 

From the recurrence relation (\ref{eqrrttj}),  we have
\[
xQ_{n}(x; t) =  Q_{n+1}(x; t) + \beta_{n+1}(t)Q_{n}(x; t)  + \alpha_{n+1}(t)xQ_{n-1}(x; t), \quad \ n \geq 1.
\]
Applying the moment function $\mathcal{L}^{(t)}$ in the above equality gives
\[ 
\tau_{n}(t) = \sigma_{n+1,n+1}(t) + \beta_{n+1} (t)\sigma_{n,n} (t) + \alpha_{n+1}(t) \tau_{n-1}(t), \quad n \geq 1.
\]
Consequently, with the observation that $\alpha_{2}(t) + \beta_1(t) = \mathcal{L}^{(t)}[x]/\mathcal{L}^{(t)}[1]$, we find 
\[
\frac{\tau_{n}(t)}{\sigma_{n,n}(t)}  = \alpha_{n+2}(t) + \beta_{n+1}(t)  + \frac{\tau_{n-1}(t)}{\sigma_{n-1,n-1}(t)} 
= \sum_{k=1}^{n+1} [\alpha_{k+1}(t)  + \beta_{k}(t)]  =  \sum_{k=1}^{n+1} \gamma_{k}(t),
\]
for $n \geq 1$, this concludes the proof of the Lemma.  \boxend 
\end{proof}

Since $Q_{n}(x; t)$ is a polynomial of degree $n$ in the variable $x$, from (\ref{Eq-Def-Lorthogonality}) we have
$
\mathcal{L}^{(t)} [x^{-n} $ $ Q_{n-1}(x; t) Q_{n}(x; t) ] = 0, $ $ n\geq1.
$
Differentiating  with respect to the variable $t$ and denoting $\frac{\partial}{\partial t} Q_{n}(x; t)=\dot{Q}_{n}(x; t)$, from (\ref{Eq-MF-TodaModification}), we obtain
\begin{equation} \label{auxiliar03}
\begin{array}{l}
 \displaystyle \mathcal{L}^{(t)} \left[ x^{-n} \dot{Q}_{n-1}(x; t)Q_{n}(x; t) \right] + \mathcal{L}^{(t)} \left[ x^{-n} Q_{n-1}(x; t)\dot{Q}_{n}(x; t) \right] \\[1.5ex]
 \displaystyle - \, \mathcal{L}^{(t)} \left[ x^{-n}\left(\pp\,x +\frac{\qq\,}{x}\right)Q_{n-1}(x; t) Q_{n}(x; t) \right] =  0, \quad n \geq 1.
\end{array}
\end{equation}
However, since $\dot{Q}_{n}(x; t)= \sum_{j=0}^{n-1} \dot{a}_{n,j}(t) x^j$ is a polynomial of degree at most $n-1$ in the variable $x$,
again from  (\ref{Eq-Def-Lorthogonality}), we can see that 
$
   \mathcal{L}^{(t)}[ x^{-n} $ $ \dot{Q}_{n-1}(x; t)Q_{n}(x; t)] = 0  
$
and
$
\mathcal{L}^{(t)}[ x^{-n} $ $ Q_{n-1}(x; t)\dot{Q}_{n}(x; t)] = \dot{a}_{n,0}(t) \sigma_{n-1,-1}(t),
$
for $n \geq 1$.   Furthermore,
\[
\mathcal{L}^{(t)} \left[ x^{-n+1}Q_{n-1}(x; t) Q_{n}(x; t) \right] =
a_{n-1,n-1}(t) \sigma_{n,n}(t) =  \sigma_{n,n}(t)
\]
and
$
\mathcal{L}^{(t)} \left[ x^{-n-1}Q_{n-1}(x; t) Q_{n}(x; t) \right] = a_{n-1,0}(t) \sigma_{n,-1}(t), 
$
for $n \geq 1$. By substituting these in (\ref{auxiliar03}), we then conclude that
\[
\dot{a}_{n,0}(t) \sigma_{n-1,-1}(t) =  \pp\,  \sigma_{n,n}(t)  + \qq\, a_{n-1,0}(t) \sigma_{n,-1}(t), \quad n \geq 1.
\]
Using (\ref{Eq-coefs-TTRR}),  (\ref{an0}) and  (\ref{sigman}), we get
\begin{equation} \label{soma_dos_betas}
\sum_{k=1}^{n} \frac{\dot{\beta}_k(t)}{\beta_k(t)}    =
-\pp\,\alpha_{n+1}(t)  + \qq\, \frac{ \alpha_{n+1}(t) }{ \beta_{n+1}(t) \beta_{n}(t)},\quad n\geq1.
\end{equation}
From this and setting the initial conditions $\beta_{0}(t)=1$ and $\alpha_{1}(t)= 0$, the relation (\ref{Eq-Toda-02}) of Theorem \ref{Thm-toda-general} holds.
 
Now from (\ref{Eq-Def-Lorthogonality}), we observe that
$
\mathcal{L}^{(t)} \left[ x^{-n}Q_{n}^2(x; t) \right] = \sigma_{n,n}(t), $ for  $ n \geq 0.
$
Differentiate $\sigma_{n,n}(t)$ with respect to $t$ and observing that 
 $\mathcal{L}^{(t)}[ x^{-n} \dot{Q}_{n}(x; t)$ $Q_{n}(x; t)] = 0$, it yields
\begin{equation} \label{auxiliar}
\dot{\sigma}_{n,n}(t)=  - \mathcal{L}^{(t)} \left[ x^{-n}\left(\pp\,x+\frac{\qq\,}{x}\right)Q_{n}^2(x; t) \right], \quad n \geq 0.
\end{equation}
Observe, from (\ref{Eq-Def-Lorthogonality}), that 
$
\mathcal{L}^{(t)} \left[ x^{-n-1}Q_{n}^2(x; t) \right] = a_{n,0}(t) \sigma_{n,-1}(t), $ for $ n \geq 0
$
and, from Lemma \ref{Lemma-01},  
\[
\mathcal{L}^{(t)} \left[ x^{-n+1}Q_{n}^2(x; t) \right] = a_{n,n-1}(t) \sigma_{n,n}(t) + \sigma_{n,n}(t) \sum_{k=1}^{n+1} [\alpha_{k+1}(t) + \beta_{k}(t)],
\]
for $n \geq 0,$ where we have taken $a_{0,-1}(t) = 0$. One can verify also from  (\ref{abzeros})  that 
\[
a_{n,n-1}(t) +  \sum_{k=1}^{n+1} [\alpha_{k+1}(t) + \beta_{k}(t)] =  \alpha_{n+2}(t) + \beta_{n+1}(t) + \alpha_{n+1}(t), \quad n \geq 0.
\]
Hence, using the above results, the equation (\ref{auxiliar}) can be given as
\[ 
-\pp\,\sigma_{n,n}(t)  \left[\alpha_{n+2}(t) + \beta_{n+1}(t) + \alpha_{n+1}(t) \right] - \qq\,  a_{n,0}(t) \sigma_{n,-1}(t)=  \dot{\sigma}_{n,n}(t),
\] 
for $n \geq 0$.  Thus, from (\ref{sigxi}) and (\ref{sigman}), we have
\[
 -\pp\, [\alpha_{2}(t) + \beta_{1}(t)+ \alpha_{1}(t)]  - \qq\, \frac{1}{\beta_{1}(t)} = \frac{\dot{\sigma}_{0,0}(t)}{\sigma_{0,0} (t)},  
\]
and, for $n \geq 1,$
\[
 -\pp\, \left[\alpha_{n+2}(t) + \beta_{n+1}(t) + \alpha_{n+1}(t)   \right] 
- \qq\, \frac{1}{\beta_{n+1}(t)} = \frac{\dot{\sigma}_{0,0}(t)}{\sigma_{0,0} (t)}+\sum_{k=2}^{n+1}\frac{\dot{\alpha}_k(t)}{\alpha_k(t)}. 
\]
Consequently the relation (\ref{Eq-Toda-01}) of Theorem \ref{Thm-toda-general} holds for $n\geq 2$. On the other hand, since $\beta_{0}(t)= 1$, $\alpha_{1}(t)=0$ and $\alpha_{0}(t)$ is arbitrary (but, we set $\alpha_{0}(t)=-1$)  the relation {\rm(\ref{Eq-Toda-01})} clearly holds for $n=1$.  

Finally, since  $\gamma_{n}(t) = \alpha_{n+1}(t) + \beta_{n}(t)$, for $n\geq1$  from (\ref{Eq-Toda-02}) and (\ref{Eq-Toda-01}), we easily obtain
equation (\ref{Eq-Toda-03}) of Theorem \ref{Thm-toda-general}. 

\setcounter{equation}{0}
\section{Lax pairs and the proof of Theorem \ref{Thm-Lpair-ETLE}}
\label{Sec-Lax-Lorth}

As in Nakamura \cite{Naka2005}, by considering the infinite matrices
\[ \hspace*{-1ex}
{\mathcal S} = \left(
\begin{array}{cccccc}
b_1    & 1      & 0      & \cdots & 0      & \cdots \\
a_2    & b_2    & 1      & \ddots & \vdots &        \\
0      & a_3    & b_3    & \ddots & 0      & \cdots \\
\vdots & \ddots & \ddots & \ddots & 1      & \ddots \\
0      & \cdots & 0      & a_n    & b_n    & \ddots \\
\vdots &        & \vdots & \ddots & \ddots & \ddots 
\end{array}\right), \quad
{\mathcal T}=\left(
\begin{array}{cccccc}
0      & 0      & \cdots & \cdots & 0      & \cdots \\
-a_2   & 0      & 0      &        & \vdots &        \\
0      & -a_3   & 0      & \ddots & \vdots &        \\
\vdots & \ddots & \ddots & \ddots & 0      & \cdots \\
0      & \cdots & 0      & -a_n   & 0      & \ddots \\
\vdots &        &\vdots  & \ddots & \ddots & \ddots 
\end{array}\right), 
\]
\noindent the Toda lattice equations (\ref{eqtod}) can also be represented in the matrix form 
\begin{equation} \label{laxpairOP}
\dot{{\mathcal S}} =  [{\mathcal S},{\mathcal T}] = {\mathcal S}\,{\mathcal T} - {\mathcal T}\,{\mathcal S}.
\end{equation}
The pair $\{{\mathcal S},{\mathcal T}\}$ is called a {\it Lax pair}  and (\ref{laxpairOP}) is called a {\it Lax representation} for the  Toda lattice (\ref{eqtod}). Another Lax pair for the  Toda lattice (\ref{eqtod}) can also be found in \cite{{Deift1999},{to1}}.

For the case of the finite relativistic Toda lattice equations (\ref{RTL1}) and (\ref{RTL2}) (i.e., $n=1, 2, \ldots, N$), we mentioned that in Suris \cite{Suris1996,Suris1997} (see also Coussement et al. \cite{CKVA2002} for a generalized form of the finite relativistic Toda lattice), by considering the bidiagonal matrices ${\mathcal M}_N$ and ${\mathcal P}_N$, given by
\[ 
\hspace*{-6ex}
{\mathcal M}_N =\left(
\begin{array}{cccccc}
\beta_1 & 1       & 0       & \cdots       & 0        \\
0       & \beta_2 & 1       &              & \vdots   \\
\vdots  & \ddots  & \ddots  & \ddots       & \vdots   \\
\vdots  &         & \ddots      &  \beta_{N-1} & 1        \\
0       & \cdots  & \cdots  & 0           & \beta_{N}
\end{array}\right), \
\mathcal {P}_{N} =\left(
\begin{array}{cccccc}
1       & 0       & 0      & \cdots    & 0      \\
-\alpha_2 & 1     & 0      &           & \vdots \\
\vdots  & \ddots  & \ddots & \ddots    & \vdots \\
\vdots  &         & \ddots & 1         & 0      \\
0       &  \cdots & \cdots& -\alpha_{N} & 1
\end{array}\right),
\]
\noindent it was shown that (\ref{RTL1}) can be written in the Lax form
\begin{equation} \label{GRTL}
\left\{\begin{array}{l}
 \dot{{\mathcal M}}_N =  {\mathcal M}_N{\mathcal A}_N-{\mathcal B}_N{\mathcal M}_N, \\[1.5ex]
 \dot{{\mathcal P}}_N =  {\mathcal P}_N{\mathcal A}_N-{\mathcal B}_N{\mathcal P}_N,
\end{array}\right.
\end{equation}
with ${\mathcal A}_N =-({\mathcal M}^{-1}_N{\mathcal P}_N)_{-}$ and $ \mathcal {B}_{N} =-({\mathcal P}_N{\mathcal M}^{-1}_N)_{-}$, where $Z_{-}$ denotes the strictly lower triangular part of $Z$. Moreover, by considering ${\mathcal A}_N =-({\mathcal P}^{-1}_N{\mathcal M}_N)_{-}$ and $ \mathcal {B}_{N} =-({\mathcal M}_N{\mathcal P}^{-1}_N)_{-}$, it was also shown that system (\ref{RTL2}) can also be written in the form (\ref{GRTL}). Notice that in neither case we have a Lax pair of the form (\ref{laxpairOP}) for the finite relativistic Toda lattice equations (\ref{RTL1}) and (\ref{RTL2}).

The aim of this section is to present a Lax pair of the form (\ref{laxpairOP}) for the extended relativistic Toda lattice equations (\ref{Eq-Toda-01}) and (\ref{Eq-Toda-03}).

Let us first consider the extended relativistic Toda lattice equations of the finite order  
\begin{equation}  \label{newTodagamma-N}
\dot{\gamma}_{n} =  \pp\left(\alpha_{n}\gamma_{n} - \alpha_{n+1} \gamma_{n+1}\right) + \qq\left( \frac{ \alpha_{n+1} }{\gamma_{n} - \alpha_{n+1}} -  \frac{ \alpha_{n} }{\gamma_{n-1} - \alpha_{n}} \right), 
\end{equation}
and
\begin{equation}  \label{newTodaxi-N}
\dot{\alpha}_{n} =   \pp \alpha_{n} \left( \alpha_{n-1} +\gamma_{n-1} - \alpha_{n}-\gamma_{n}\right) 
+  \qq \alpha_{n} \left( \frac{1}{\gamma_{n-1} - \alpha_{n}} - \frac{1}{\gamma_{n} - \alpha_{n+1}} \right),
\end{equation}
for $n =1,2, \ldots, N$,
where $\gamma_{n}(t) = \alpha_{n+1}(t) + \beta_{n}(t)$, $\beta_{0}(t)=1$, $\alpha_{0}(t)= -1$  and $\alpha_{1}(t)=0$, with the additional assumption $\alpha_{N+1}(t) = 0$. Observe that with $\alpha_{N+1}(t)= 0$ the product $\alpha_{N+1}(t) \gamma_{N+1}(t)$ that appears in (\ref{newTodagamma-N}) is also zero. 

For $N \geq 1$, let the  $N \times N$ matrices $\mathcal{H}_N = \mathcal{H}_N(t)$, $\mathcal{X}_{N} = \mathcal{X}_{N}(t)$ and $\mathcal{Y}_{N} = \mathcal{Y}_{N}(t)$ be given by 
\[
{\mathcal H}_N=\!\left(
\begin{array}{cccccc}
\gamma_1 & \gamma_2 & \gamma_3 & \cdots     & \gamma_{N-1} & \gamma_{N} \\
\alpha_2   & \gamma_2 & \gamma_3 & \cdots     & \gamma_{N-1} & \gamma_{N} \\
0        & \alpha_3   & \gamma_3 & \cdots     & \gamma_{N-1} & \gamma_{N} \\
\vdots   & \ddots   & \ddots   & \ddots     & \vdots       & \vdots     \\
0        & \cdots   & 0        & \alpha_{N-1} & \gamma_{N-1} & \gamma_{N} \\
0        & \cdots   & 0        & 0          & \alpha_{N}     & \gamma_{N}
\end{array}\right),  \quad
\mathcal {X}_{N}=\!\left(\!
\begin{array}{cccccc}
\alpha_1  & 0       & 0      & \cdots      & 0          & 0      \\
-\alpha_2 & \alpha_2  & 0      & \cdots      & 0          & 0      \\
 0      & -\alpha_3 & \alpha_3 & \ddots      & \vdots     & \vdots \\
\vdots  & \ddots  & \ddots & \ddots      & 0          & 0      \\
0       & \cdots  & 0      & -\alpha_{N-1} & \alpha_{N-1} & 0      \\
0       & \cdots  & 0      & 0           & -\alpha_{N}  & \alpha_{N}
\end{array}\right)
\]
and
\[
\mathcal {Y}_{N} =\left(
\begin{array}{cccccc}
\frac{1}{\gamma_1-\alpha_2} & -\frac{1}{\gamma_1-\alpha_2} & 0 &\cdots &      0 &   0     \\[0ex]
0  & \frac{1}{\gamma_2-\alpha_3}   & -\frac{1}{\gamma_{2}-\alpha_{3}}& \ddots & \vdots & \vdots  \\[0ex]
 0  &   0  &  \ddots  &   \ddots         & 0 & 0      \\
 \vdots  &   \ddots  & \ddots    &   \frac{1}{\gamma_{N-2}-\alpha_{N-1}}         & -\frac{1}{\gamma_{N-2}-\alpha_{N-1}}  & 0      \\[1ex]
   0   &  \cdots   & 0   & 0 &    \frac{1}{\gamma_{N-1}-\alpha_{N}}   & -\frac{1}{\gamma_{N-1}-\alpha_{N}}        \\[1ex]
    0   &  \cdots    & 0   &     0 & 0 & \frac{1}{\gamma_{N}-\alpha_{N+1}}  
\end{array}\right). 
\]
The matrix $\mathcal{H}_N$ has already been shown to be interesting in the studies related to L-orthogonal polynomials.  From results first appeared in \cite{SR-A-CATCF92} (see also \cite{daSilvaRanga2005,SRVA2002} on further studies), the zeros of $Q_N(x; t)$ are exactly the eigenvalues of the Hessenberg matrix  $\mathcal{H}_N(t)$. 

Now with the above matrices by performing the respective matrix multiplications one easily finds that 
\[
\dot{\mathcal{H}}_{N} = \pp \left(\mathcal{H}_{N} \mathcal{X}_{N} - \mathcal{X}_{N} \mathcal{H}_{N}\right)  + \qq \left(\mathcal{H}_{N} \mathcal{Y}_{N} - \mathcal{Y}_{N} \mathcal{H}_{N}\right),  
\] 
for any $N \geq 1$.  Hence, we can state the following theorem. 

\begin{theorem} \label{Thm-Finite-Lpair-ETLE}  
A Lax representation for the  extended relativistic Toda lattice equations of finite order $N$ $(N \geq 1)$ given by $(\ref{newTodagamma-N})$  and $(\ref{newTodaxi-N})$ is 
\[
 \dot{\mathcal{H}}_{N}  = [\mathcal{H}_{N}, \mathcal{F}_{N}] = \mathcal{H}_{N} \mathcal{F}_{N} - \mathcal{F}_{N} \mathcal{H}_{N},
\]
where $\mathcal{F}_{N} = \pp \mathcal{X}_{N} + \qq \mathcal{Y}_{N}$, and given in $(\ref{Eq-LaxPairHF})$.
\end{theorem}

To prove Theorem \ref{Thm-Lpair-ETLE} one only needs to let  $N \to \infty$ in Theorem \ref{Thm-Finite-Lpair-ETLE}. 

\setcounter{equation}{0}
\section{From L-orthogonal polynomials on the positive real axis}
\label{Sec-LOrthogonal-RL}

In this section we consider the case in which the moment functional $\mathcal{L}$ is given by 
\[
      \mathcal{L}[f] = \int_{a}^{b} f(x) d \psi(x),
\]
where $0 \leq  a < b \leq \infty$ and $\psi$ is a strong positive measure defined on $[a, b]$.  With the term ``strong'' 	we mean that the moments $\mathcal{L}[x^n] = \nu_{n}^{(0)}$ exists for all $n \in \mathbb{Z}$. The existence of the moments $\mathcal{L}^{(t)}[x^n] = \nu_{n}^{(t)}$, for all $n \in \mathbb{Z}$, depends on the choice of $\pp$ and $\qq$, especially if $a=0$ and/or $b=\infty$. 

Clearly, the choice $\pp> 0$ and $\qq>0$, that we will assume throughout in this section, guarantees the existence of all the moments since it is easily verified that 
\[
   \int_{a}^{b} x^{n} e^{-t\left(\pp x + \frac{\qq}{x}\right)}  d \psi(x) 
   \leq e^{-2t\sqrt{\pp \qq}} \int_{a}^{b} x^{n}   d \psi(x).    
\]
With the existence of the moments,  the determinantal conditions in (\ref{Eq-newHankelCond}) also hold since $e^{-t\left(\pp x + \qq/x\right)}  d \psi(x)$ leads to a positive measure in $[a,b]$. To be precise, we have $ H_{n}^{(-n)}(t) > 0 $ and $ H_{n+1}^{(-n)}(t) > 0,$
for $n \geq 0$. We can also, without any loss of generality, consider $\mathcal{L}^{(t)}$ in the form 
\begin{equation} \label{Eq-MF-TodaModification-PRL}
    \mathcal{L}^{(t)}[f(x)] = \mathcal{L}[f(x) e^{-t\left(x + \frac{\qq}{x}\right)}],  
\end{equation}
with $\qq > 0$, by absorbing the positive $\pp$ into the parameter $t$. 

If we consider the L-orthogonal polynomials $Q_{n}(x; t)$ with respect to this moment functional  then from the positiveness of  $e^{-t\left(x + \qq/x\right)}  d \psi(x)$ one can  state that (see \cite{JTW1980})  the coefficients  in the associated recurrence relation (\ref{eqrrttj}) satisfy  $\beta_{n}(t) > 0$ and  $ \alpha_{n+1}(t) > 0, \ n \geq 1.$

On the other hand,  by Theorem \ref{Thm-toda-general}  these coefficients satisfy the extended relativistic Toda lattice equations 
\begin{equation} \label{Eq-Toda-02-PRL}
\dot{\beta}_{n} =   \beta_{n} \left( \alpha_{n} - \alpha_{n+1} \right) +
\qq\, \beta_{n} \left(\frac{\alpha_{n+1}}{\beta_{n+1}\beta_{n}} - \frac{\alpha_{n}}{\beta_{n}\beta_{n-1}} \right)
\end{equation}
and
\begin{equation} \label{Eq-Toda-01-PRL}
\dot{\alpha}_{n} =    \alpha_{n} \left( \alpha_{n-1} + \beta_{n-1}-\alpha_{n+1} - \beta_{n} \right)
+ \qq\,  \alpha_{n} \left( \frac{1}{\beta_{n-1}} - \frac{1}{\beta_{n}} \right),
\end{equation}
for $n \geq 1$, with the initial conditions $\beta_{0}(t)= 1$, $\alpha_{0}(t)= -1$  and $\alpha_{1}(t)= 0$. 

We now analyse results corresponding to such $\mathcal{L}^{(t)}$ with two particular examples.  \\[0ex] 

\noindent {\bf Example 1}\ \ For $\delta > 0$, let the moment functional $\mathcal{L}$ be given by 
$
  \mathcal{L}[f(x)] = \int_{0}^{\infty} f(x) \, d \psi(x),
$
where  $d \psi(x) = x^{-\frac{1}{2}} e^{-\delta(x + \qq/x)} dx$. Then, the moment functional $\mathcal{L}^{(t)}$ defined as in (\ref{Eq-MF-TodaModification-PRL}) satisfies  
\[ 
  \mathcal{L}^{(t)}[f(x)] = \int_{0}^{\infty} f(x) \, d \psi^{(t)}(x),
\]
where $d \psi^{(t)}(x) = x^{-\frac{1}{2}} e^{-(t+\delta)(x + \qq/x)} dx$.  

Considering the L-orthogonal polynomials $Q_{n}(x; t)$ with respect to this moment functional  we find that the coefficients of the associated three term recurrence relation (\ref{eqrrttj}) satisfy 
\begin{equation} \label{Eq-Example-1}
     \beta_{n}(t) = \sqrt{\qq} \quad \mbox{and} \quad \alpha_{n+1}(t) = \frac{n}{2(t+\delta)}, \quad n \geq 1.
\end{equation}
This follows from results given in \cite{Ranga-PAMS1995} (p. $3139$).

Substitution of the values (\ref{Eq-Example-1}) in the right hand sides of (\ref{Eq-Toda-02-PRL}) and (\ref{Eq-Toda-01-PRL}) we find  $ \dot{\beta}_{n} = 0,$   for $n \geq 1$,  and
\[ 
\dot{\alpha}_{n} = \displaystyle \alpha_{n} \left( \alpha_{n-1} + \beta_{n-1}-\alpha_{n+1} - \beta_{n} \right)
+ \qq\,  \alpha_{n} \left( \frac{1}{\beta_{n-1}} - \frac{1}{\beta_{n}} \right) = 
-\frac{n-1}{2(t+\delta)^2}, \quad n \geq 2. 
\]
The values obtained here for $\dot{\beta}_{n}$ and $\dot{\alpha}_{n}$ are what we obtain by direct differentiation in (\ref{Eq-Example-1}). 

It turns out that the measure $\psi$ in Example 1 is such that 
\begin{equation} \label{Eq-symmetry1}
    \frac{d \psi(x)}{\sqrt{x}} = - \frac{d \psi(\qq/x)}{\sqrt{\qq/x}}.
\end{equation}
Consequently, the measure $\psi^{(t)}$ in Example 1 also satisfies the symmetric property (\ref{Eq-symmetry1}).

It is known  (see, for example, \cite{Ranga-PAMS1995,SRAM}) that under this symmetric property the coefficients $\beta_{n}$ in the three term recurrence (\ref{eqrrttj}) must satisfy $\beta_{n}(t) = \sqrt{\qq}$, $n \geq 1$. 

In general, if we start with any strong measure $\psi$ that satisfies the symmetry (\ref{Eq-symmetry1}) and then proceed to create the linear functional $\mathcal{L}$ and to build the moment functionals $\mathcal{L}^{(t)}$ as in (\ref{Eq-MF-TodaModification-PRL}), then $\beta_{n}(t) =  \sqrt{\qq}$, $n \geq 1$.  Moreover,  for $\alpha_{n}(t)$ we obtain $\alpha_{n}(t) > 0$, $n \geq 2$ and from (\ref{Eq-Toda-01-PRL})  
\[
 \dot{\alpha}_{n}(t) = \alpha_{n}(t)  \left[\alpha_{n-1}(t) - \alpha_{n+1}(t)  \right], \quad n \geq 2,
\]
with $\alpha_1(t) =0$, which is known as Langmuir or Volterra lattice (see \cite{Peh2001}). \\

\noindent{\bf Example 2}\ \ For $\delta > 0$, let the moment functional $\mathcal{L}$ be given by 
$ 
  \mathcal{L}[f(x)] = \int_{0}^{\infty} f(x) \, d \tilde{\psi}(x),
$
where  $d \tilde{\psi}(x) = (x + \sqrt{\qq}) x^{-\frac{3}{2}} e^{-\delta(x + \frac{\qq}{x})} dx$. Then the moment functional $\mathcal{L}^{(t)}$ defined as in (\ref{Eq-MF-TodaModification-PRL}) satisfies  
\begin{equation} \label{Eq-Exemple-2}
  \mathcal{L}^{(t)}[f(x)] = \int_{0}^{\infty} f(x) \, d \tilde{\psi}^{(t)}(x),
\end{equation}
where $d \tilde{\psi}^{(t)}(x) = (x + \sqrt{\qq}) x^{-\frac{3}{2}} e^{-(t+\delta)(x + \frac{\qq}{x})} dx$.  

Let us denote the coefficients in the three term recurrence relation  (\ref{eqrrttj}) with respect to the moment functional $\mathcal{L}^{(t)}$ in (\ref{Eq-Exemple-2}) as $\tilde{\alpha}_{n}(t)$ and $\tilde{\beta}_{n}(t)$. The measures $\tilde{\psi}^{(t)}$ can be verified to satisfy the symmetry 
\[
   d\tilde{\psi}^{(t)}(x) = - d\tilde{\psi}^{(t)}(\qq/x),
\]
and further $d\tilde{\psi}^{(t)}(x)  = \frac{x+\sqrt{\qq}}{x} d\psi^{(t)}(x)$, where $\psi^{(t)}$ are measures given in Example 1. Consequently, using results found in \cite{SR96}, we obtain 
\[
  \tilde{\beta}_{n} (t) = \frac{l_{n-1}(t)}{l_{n}(t)}, \quad \tilde{\alpha}_{n+1}(t) =  \tilde{\beta}_{n}(t)\,[l_{n}^2(t) - 1] , \quad n \geq 1, 
\]
where $ l_{n}(t) = \displaystyle{1 + \frac{n/[2\sqrt{\qq}(t + \delta)]}{l_{n-1}(t)+1}}, \; n \geq 1 \; $ and $\;  l_{0}(t)=1$.  These values for $\tilde{\beta}_{n}$ and $\tilde{\alpha}_{n}$, together with the values for $\dot{\tilde{\beta}}_{n}$ and $\dot{\tilde{\alpha}}_{n}$ obtained from these, can be successively substituted in 
\[
  \begin{array}{l}
   \displaystyle \dot{\tilde{\beta}}_{n} =   \tilde{\beta}_{n} \left( \tilde{\alpha}_{n} - \tilde{\alpha}_{n+1} \right) +
 \qq\, \tilde{\beta}_{n} \left(\frac{\tilde{\alpha}_{n+1}}{\tilde{\beta}_{n+1}\tilde{\beta}_{n}} - \frac{\tilde{\alpha}_{n}}{\tilde{\beta}_{n}\tilde{\beta}_{n-1}} \right), \\[3ex]
  \displaystyle \dot{\tilde{\alpha}}_{n} =   \tilde{\alpha}_{n} \left( \tilde{\alpha}_{n-1} + \tilde{\beta}_{n-1}-\tilde{\alpha}_{n+1} - \tilde{\beta}_{n} \right)
+  \qq\,  \tilde{\alpha}_{n} \left( \frac{1}{\tilde{\beta}_{n-1}} - \frac{1}{\tilde{\beta}_{n}} \right), 
  \end{array} \quad  n \geq 1, 
\]
to verify the validity of these extended relativistic Toda lattice equations.  


\setcounter{equation}{0}
\section{From kernel polynomials on the unit circle}
\label{Sec-CDKernels-UC}

Let  $\mu$ be a positive measure defined on the unit circle $\mathbb{T}=\{z=e^{i\theta} : 0 \leq \theta \leq 2\pi \}$ and let 
\[
  \mathcal{L}[f] = \int_{\mathbb{T}} f(z) (z-w)d \mu(z) \quad \mbox{and}  \quad \mathcal{L}^{(t)}[f] = \int_{\mathbb{T}} f(z) (z-w)d\mu^{(t)}(z),
\]
where 
\[
    d\mu^{(t)}(z) = e ^{-t \left(\overline{\qq}z+ \qq/z \right)}   d \mu(z) = e^{-2t[\Re(\qq)\cos \theta + \Im(\qq) \sin \theta]} d \mu(e^{i\theta}).   
\]
For convention and also for convenience  we have replaced $x$ by $z$.

Clearly, $\mu^{(t)}$ is a well defined positive measure on the unit circle  for any $t$ real. Thus, from now on consider $t \in (-\infty, \infty)$. We denote the $n^{th}$ degree monic orthogonal polynomial and  orthonormal polynomial  associated with  the measure $\mu^{(t)}$ by $\Phi_{n}(z; t)$ and $\varphi_{n}(z; t)$, respectively (see \cite{Simon-book-p1}).  We also denote the associated Verblunsky coefficients by $\Va_{n}(t)$. That is,
$
     \Va_{n}(t) = - \overline{\Phi_{n+1}(0; t)},$ $n \geq 1. 
$

Now, with a fixed $w$ such that $|w|=1$ and fixed $t$, we consider the L-orthogonal polynomials $Q_{n}(z;t)$ defined  by (\ref{Eq-Def-Lorthogonality}) with $z$ in the place of $x$. The existence of these polynomials and that they satisfy the three term recurrence  (\ref{eqrrttj}), i.e., 
\begin{equation} \label{eqrrttj-CDkerenl1}
    Q_{n+1}(z; t) = [z - \beta_{n+1}(t)] Q_{n}(z; t) - \alpha_{n+1}(t)z Q_{n-1}(z; t), \quad n \geq 1,
\end{equation}
with $Q_{0}(z; t) = 1$ and $Q_{1}(z; t) = z- \beta_{1}(t)$, follow from results given in \cite{CostaFelixRanga-JAT2013}.

These polynomials are actually the monic kernel polynomials (or monic CD kernel as in Simon \cite{Simon-book-p1}) with respect to the measure $\mu^{(t)}$. Precisely, we have
\[
  \kappa_{n} \overline{\varphi_{n}(w;t)}\, Q_{n}(z; t) =  K_{n}^{(t)}(z,w) = \sum_{j=0}^{n} \overline{\varphi_{j}(w;t)} \,\varphi_{j}(z;t), \quad n \geq 1,
\]
where $\kappa_{n}^{-2} = \mu_{0}^{(t)} \prod_{j=0}^{n-1}(1-|\Va_{j}(t)|^2)$ and $\mu_{0}^{(t)} =\int_{\mathbb{T}} d\mu^{(t)}(z)$.  Moreover, for the coefficients $\beta_{n}(t) = \beta_{n}(w,t)$ and $\alpha_{n}(t) = \alpha_{n}(w,t)$ we have (see \cite{CostaFelixRanga-JAT2013} (Thm. $2.1$)), 
\[
   \beta_{n}( t) = - \frac{\rho_{n}^{(t)}(w)}{\rho_{n-1}^{(t)}(w)}, \quad \alpha_{n+1}(t) = [1+\rho_{n}^{(t)}(w) \Va_{n-1}(t)]\,[1 - \overline{w\rho_{n}^{(t)}(w) \Va_{n}(t)}\,] w,
\]
for $n \geq 1$, where $\rho_{n}^{(t)}(w) = \Phi_{n}(w;t)/ \Phi_{n}^{\ast}(w;t)$, $n \geq 0$, and $\Phi_{n}^{\ast}(w;t) = w^{n} \overline{\Phi_n(1/\bar{w};t)}$ denotes the reciprocal polynomial of $\Phi_n(w;t)$.

Clearly, by Theorem \ref{Thm-toda-general} the coefficients $\beta_n(t)$ and $\alpha_{n}(t)$ satisfy the extended relativistic Toda lattice equations (\ref{Eq-Toda-02}) and (\ref{Eq-Toda-01}), with $\pp = \overline{\qq}$. 

Observe that, different from the results presented in Section \ref{Sec-LOrthogonal-RL}, the  coefficients $\beta_n(t)$ and $\alpha_{n}(t)$ now are complex valued. However, by taking $w=1$ we can write the three term recurrence (\ref{eqrrttj-CDkerenl1}) in a different form which involves only real coefficients. 

Let 
\begin{equation} \label{parametricSeq1}
 g_{n}(t) = \frac{1}{2} \frac{\big|1 - \rho_{n-1}^{(t)} \Va_{n-1}(t)\big|^2} 
     {\big[1 - \Re[\rho_{n-1}^{(t)}\Va_{n-1}(t)]\big]}, \quad n \geq 1,
\end{equation}
where $\rho_{j}^{(t)} = \rho_{j}^{(t)}(1)$. It is easy to check that all $g_n(t) \in (0,1)$, hence the terms of the following sequence are all positive
\[
    \xi_n(t) = \xi_0(t) \prod_{j=1}^{n}\left(1-g_j(t)\right), \quad n \geq 1, \qquad \xi_0(t) :=   \int_{\mathbb{T}} d\mu^{(t)}(z).
\]
With this notation we introduce the normalized CD kernels by
%
$    R_n(z; t) := \xi_n(t)$ $K_{n}^{(t)}(z,1),$ $n \geq 0. $

It turns out (see \cite{CostaFelixRanga-JAT2013}) that these kernel polynomials, $R_n(z; t)$, satisfy the following three term recurrence relation
\begin{equation} \label{Eq-TTRR-Rn1}
R_{n+1}(z; t) = [(1+ic_{n+1}(t))z + (1-ic_{n+1}(t))]\,R_{n}(z; t) - 4d_{n+1}(t) z\,R_{n-1}(z; t), 
\end{equation}
for $n \geq 1$, with $R_{0}(z; t) = 1$ and $R_{1}(z; t)=(1+ic_{1}(t))z+(1-ic_{1}(t))$, where both $\{c_n(t)\}_{n \geq 1}$ and $\{d_{n+1}(t)\}_{n \geq 1}$ are  real sequences. In fact,
%
\[
c_{n}(t) = \frac{\Im (\rho_{n-1}^{(t)}\Va_{n-1}(t))} {\Re(\rho_{n-1}^{(t)}\Va_{n-1}(t))-1} \in \mathbb{R}
\quad \mbox{and} \quad 
d_{n+1}(t) = [1-g_{n}(t)] g_{n+1}(t), \quad n \geq 1,
\]
with $g_n(t)$ given by (\ref{parametricSeq1}). In the standard terminology, this means that $\{d_{n+1}(t)\}_{n \geq 1}$ is a positive chain sequence for any $t$, and $\{g_{n+1}(t)\}_{n\geq 0}$  is a parameter sequence  for  $\{d_{n+1}(t)\}_{n \geq 1}$ (for more details on chain sequences see, for example, \cite{Chihara-book}).  

Now, by using Theorem \ref{Thm-toda-general}, we can state the following.

\begin{theorem}\label{todalortcor} 
The coefficients $c_{n}(t)$ and $d_{n}(t)$ of the three term recurrence relation $(\ref{Eq-TTRR-Rn1})$ satisfy
\[ 
\begin{array}{rcl}
   \dot{c}_{1} &=& \displaystyle  -4\Re(\qq) \left[\frac{d_2\left(c_1+c_{2}\right)}{1+c^{2}_{2}}\right] - 4\Im(\qq)\left[\frac{d_2\left(1-c_{1}c_{2}\right)}{1+c^{2}_{2}}\right]
\end{array}
\] 
and, for $n \geq 2$, 
\[ 
\begin{array}{rcl}
   \dot{c}_{n} &=& \displaystyle  4\Re(\qq)\left[\frac{d_n\left(c_n +c_{n-1}\right)}{1+c^2_{n-1}}-\frac{d_{n+1}\left(c_n+c_{n+1}\right)}{1+c^2_{n+1}}\right] \\[3.5ex]
& & \displaystyle  + \ 4\Im(\qq)\left[\frac{d_n\left(1-c_nc_{n-1}\right)}{1+c^2_{n-1}}-\frac{d_{n+1}\left(1-c_nc_{n+1}\right)}{1+c^2_{n+1}}\right], 
\end{array}
\]
and 
\[
\begin{array}{rcl}
\dot{d}_{n} & = & \displaystyle   4\Re(\qq)\left[\frac{d_nd_{n-1}}{1+c^2_{n-2}} -\frac{d_nd_{n+1}}{1+c^2_{n+1}} + \frac{d_n\left(1-d_n\right)\left(c^2_{n-1}-c^2_{n}\right)}{\left(1+c^2_{n}\right)\left(1+c^2_{n-1}\right)} \right] \\[3.5ex]
& & \displaystyle  - \ 4\Im(\qq) \left[\frac{d_n d_{n-1} c_{n-2}}{1+c^2_{n-2}} - \frac{d_n d_{n+1} c_{n+1}}{1+c^2_{n+1}} + \frac{d_n\left(1-d_n\right)\left(c_{n}-c_{n-1}\right)\left(1-c_{n}c_{n-1}\right)}{\left(1+c^2_{n}\right)\left(1+c^2_{n-1}\right)}\right],
\end{array}
\]
with $c_{0}(t) = 1$ and $d_{1}(t) = 0$.  Here, we have also omitted the time variable $t$.
\end{theorem}

\begin{proof}
Since, $Q_{n}(z; t) \prod_{k=1}^{n}(1+ic_{k}(t)) = R_{n}(z; t)$, $n \geq 1$, the coefficients $\beta_{n}(t)$, $\alpha_{n}(t)$, $c_{n}(t)$ and $d_{n}(t)$ which appear in the three term recurrence relations (\ref{eqrrttj-CDkerenl1}) and (\ref{Eq-TTRR-Rn1}) are such that
\begin{equation}\label{betan-xin-for-cn-dn}
\beta_{n}(t)=-\frac{1-ic_{n}(t)}{1+ic_{n}(t)}\quad\mbox{and}\quad\alpha_{n}(t)=\frac{4d_{n}(t)}{(1+ic_{n}(t))(1-ic_{n-1}(t))},
\end{equation}
for $n\geq1,$ with $c_{0}(t)=1$ and $d_{1}(t)= 0$.
Notice that differentiating $\beta_{n}(t)$ given by (\ref{betan-xin-for-cn-dn}) with respect to $t$, we obtain
\begin{equation}\label{deriv-betan}
\frac{\dot{\beta}_{n}(t)}{\beta_{n}(t)}=-i\frac{2\dot{c}_{n}(t)}{1+c_{n}^2(t)},\quad n\geq1.
\end{equation}

Consequently since the coefficients $\beta_{n}(t)$ and $\alpha_{n}(t)$ satisfy (\ref{Eq-Toda-02}), using (\ref{betan-xin-for-cn-dn}) and (\ref{deriv-betan}), we conclude that, for $n\geq1$,
\[
\begin{array}{rcl}
\hspace*{-3ex} \dot{c}_{n}(t)&=&\displaystyle 2(\pp+\qq) \left\{\frac{d_{n}(t)[c_n(t)+c_{n-1}(t)]}{1+c^2_{n-1}(t)}-\frac{d_{n+1}(t)[c_n(t)+c_{n+1}(t)]}{1+c^2_{n+1}(t)} \right\}\\[3ex]
& & \displaystyle  + i \, 2(\pp-\qq) \left\{\frac{d_{n}(t)[1-c_n(t)c_{n-1}(t)]}{1+c^2_{n-1}(t)}-\frac{d_{n+1}(t)[1-c_n(t)c_{n+1}(t)]}{1+c^2_{n+1}(t)}\right\}.
\end{array}
\]
Hence,  the expression for $\dot{c}_{n}(t)$ in Theorem \ref{todalortcor} is a consequence of  $\pp = \overline{\qq}$. In a similar manner, using the value of $\alpha_{n}(t)$ given by (\ref{betan-xin-for-cn-dn}), one can also prove the expression for $\dot{d}_{n}(t)$.   \boxend 
\end{proof}

To obtain a special case of the results presented in Theorem \ref{todalortcor}, we assume that the measure $\mu$ satisfies the symmetry $d\mu(e^{i\theta}) = -d\mu(e^{i(2\pi-\theta)})$ and that $\qq$ is real. Hence, the measure $\mu^{(t)}$ also satisfies the same symmetry and, as a consequence, $\Va_{n-1}(t)$  is real   and $ \rho_{n}^{(t)}(1) = 1,   n \geq 1. $
This leads to the following corollary of Theorem \ref{todalortcor}.

\begin{corollary}
If $\mu$ is such that $d\mu(e^{i\theta}) = -d\mu(e^{i(2\pi-\theta)})$ and $\qq$ is real then $c_{n}(t) = 0$ for $n \geq 1$, and $d_n(t)$ satisfy 
\begin{equation}\label{langd}
 \dot{d}_{n} = \displaystyle 4\qq\, d_{n} \left(d_{n-1} -d_{n+1} \right), \quad n \geq 2, 
\end{equation}
with $d_{1}(t)= 0$.
\end{corollary}
Notice that the relation obtained in (\ref{langd}) is a type of Langmuir lattice.

\setcounter{equation}{0}
\section{From orthogonal polynomials on the unit circle}
\label{Sec-OPUC}

Let  $\mu$ be a positive measure defined on the unit circle $\mathbb{T}$ and let 
\[
  \mathcal{L}[f] = \int_{\mathbb{T}} f(z) zd \mu(z) \quad \mbox{and}  \quad \mathcal{L}^{(t)}[f] = \int_{\mathbb{T}} f(z) zd\mu^{(t)}(z),
\]
where 
\[
  d\mu^{(t)}(z) = e ^{-t \left(\overline{\qq}z+ \qq/z \right)}   d \mu(z) = e^{-2t[\Re(\qq)\cos \theta + \Im(\qq) \sin \theta]} d \mu(e^{i\theta}).   
\]
Clearly, in this case the monic L-orthogonal polynomials $Q_n(z;t)$ defined by (\ref{Eq-Def-Lorthogonality}) (with $z$ in the place of $x$) for any $t$ fixed are actually the orthogonal polynomials  on the unit circle, $\Phi_{n}(z;t)$, with respect to the measure $\mu^{(t)}$.  Again, we will denote the associated Verblunsky coefficients and the associated reciprocal polynomials of $\Phi_n(w;t)$, respectively, by $\Va_{n}(t)$ and $\Phi_{n}^{\ast}(w;t)$, $n \geq 0$. It is known (see \cite{Simon-book-p1}) that the polynomials $Q_n(z;t) = \Phi_{n}(z;t)$  satisfy the relations
\begin{equation} \label{relaOPUC}
\begin{array}{l}
  Q_n(z;t) =  z Q_{n-1}(z;t) - \overline{\Va_{n-1}(t)}\, Q_{n-1}^{\ast}(z;t), \\[1.5ex]
   Q_n(z;t) = (1 - |\Va_{n-1}(t)|^2) z Q_{n-1}(z;t) - \overline{\Va_{n-1}(t)}  Q_{n}^{\ast}(z;t),
\end{array}   \quad n \geq 1,
\end{equation}
with $ Q_0(z;t)=1$ and $Q_{0}^{\ast}(z;t)=1$.

The existence of those polynomials are guaranteed by the positiveness of the measure $\mu^{(t)}$.  That is, with the moments 
\begin{equation} \label{Eq-Moments-OPUC}
    \nu_{k}^{(t)} = \mathcal{L}^{(t)}[z^{k}],  \quad k=0, \pm1, \pm2, \ldots , 
\end{equation}
the associated Hankel determinants satisfy condition (a) of (\ref{Eq-newHankelCond}). However, in general one can not assure condition (b) of (\ref{Eq-newHankelCond}). 

It is well known that the sequence of the Verblunsky coefficients $\Va_{n}(t)$ associated with the measure $  d\mu^{(t)}(z) = e ^{t \left(z+ 1/z \right)} d \mu(z) $ satisfies the so-called Schur flow equation
\[
     \dot{\Va}_{n} = (1-|\Va_{n}|^2)\left(\Va_{n+1} - \Va_{n-1}\right), 
\]
see, for example, \cite{AmmarGragg,gol2006,muka-Naka2002}. In \cite{gol2006} a Lax representation for the Schur flow equations  was obtained by using the CMV matrices (see \cite{CMV2003} for more details about CMV matrices).

In \cite{Nenciu2005} the connection between of the Verblunsky coefficients and the defocusing Ablowitz-Ladik system
\[
    i \dot{\Va}_{n} = \Va_{n+1} -2\Va_{n} + \Va_{n-1} - |\Va_{n}|^2\left(\Va_{n+1} - \Va_{n-1}\right).
\]
was established.  Moreover, also in  \cite{Nenciu2005}, by using the CMV matrices and by defining certain Hamiltonians, a Lax representation for the defocusing Ablowitz-Ladik system was given, with emphasis for the {\em p}--periodic Verblunsky coefficients, i.e., when $\Va_{n+p} = \Va_{p}.$ \\

\noindent {\bf Assumption}:  Let the moments (\ref{Eq-Moments-OPUC}) be such that  condition (b) of (\ref{Eq-newHankelCond}) also hold. \\[-1ex]

With this assumption we also have $\Va_{n}(t) \neq 0$, $n \geq 0$. In this case, using (\ref{relaOPUC}), we can see that the polynomials $Q_n(z;t)$  satisfy the three term recurrence relation  
\begin{equation} \label{Eq-TTRR-OPUC}
  Q_{n+1}(z; t) = [z - \beta_{n+1}(t)] Q_{n}(z; t) - \alpha_{n+1}(t)\,z Q_{n-1}(z; t), \quad n \geq 1,
\end{equation}
with $Q_{0}(z; t) = 1$ and $Q_{1}(z; t) = z- \beta_{1}(t)$, where $\beta_{1}(t) = \overline{\Va_{0}(t)}$, 
\[
  \beta_{n+1}(t) =  -\frac{\overline{\Va_{n}(t)}}{\overline{\Va_{n-1}(t)}} \quad \mbox{and} \quad \alpha_{n+1}(t) = \frac{\overline{\Va_{n}(t)}}{\overline{\Va_{n-1}(t)}} (1 - |\Va_{n-1}(t)|^2), \quad n \geq 1.
\]
Since  $\overline{\Va_{n}(t)} = (-1)^{n} \beta_{1}(t) \beta_{2}(t) \cdots \beta_{n+1}(t)$, observe that
\[
  \frac{\overline{\dot{\Va}_{n}(t)}}{\overline{\Va_{n}(t)}} = \sum_{j=1}^{n+1} \frac{\dot{\beta}_{j}(t)}{\beta_{j}(t)}, \quad n \geq 0. 
\]
Thus, from (\ref{soma_dos_betas}), we find
$
  \overline{\dot{\Va}_{n}(t)} = (1-|\Va_{n}(t)|^2)\left[\qq\,\overline{\Va_{n-1}(t)} - \overline{\qq}\, \overline{\Va_{n+1}(t)}\right], $ $ n \geq 1. 
$

Hence, we can state the following result.


\begin{theorem} \label{Thm-Schur}
Let $\mu$ be a positive measure on the unit circle, and let $\{\Va_n(t)\}_{n=0}^{\infty}$ be the Verblunsky coefficients associated with the measure $\mu^{(t)}$ given by 
$ d \mu^{(t)}(z) = e ^{-t \left(\overline{\qq}z + \frac{\qq}{z}\right)}$  $d \mu(z).$
Assuming that the Verblunsky coefficients are all different from zero we obtain the following
\begin{equation} \label{Eq-Schur}
     \dot{\Va}_{n} = (1-|\Va_{n}|^2)\left(\overline{\qq}\,\Va_{n-1} - \qq\, \Va_{n+1}\right), \quad n \geq 1. 
\end{equation}
\end{theorem}

\begin{remark}
Because of the approach used in this manuscript, we had to assume that the all the Verblunsky coefficients $\Va_{n}(t)$ are different from zero to obtain the results of Theorem {\rm\ref{Thm-Schur}}.  But, this restriction is not necessary in the equation $(\ref{Eq-Schur})$.  

The system of nonlinear difference differential equations $(\ref{Eq-Schur})$ also satisfies the Schur flow equation. 
\end{remark}


\end{document}